\documentclass{lms}

\newtheorem{theorem}{Theorem}[section] 
\newtheorem{lemma}[theorem]{Lemma}     
\newtheorem{corollary}[theorem]{Corollary}

\newnumbered{assertion}{Assertion}    
\newnumbered{conjecture}{Conjecture}  
\newnumbered{definition}{Definition}
\newnumbered{hypothesis}{Hypothesis}
\newnumbered{remark}{Remark}
\newnumbered{note}{Note}
\newnumbered{observation}{Observation}
\newnumbered{problem}{Problem}
\newnumbered{question}{Question}
\newnumbered{algorithm}{Algorithm}
\newnumbered{example}{Example}
\newunnumbered{notation}{Notation} 

\usepackage{amsrefs}
\usepackage{enumerate}
\usepackage{mathbbol}
\usepackage{amssymb}
\usepackage{braket}
\usepackage{longtable}
\usepackage[colorlinks]{hyperref}
\usepackage{caption}
\usepackage{mathtools}
\usepackage{microtype}

\newcommand{\N}{\mathbb{N}}
\newcommand{\Z}{\mathbb{Z}}
\newcommand{\R}{\mathbb{R}}  

\newcommand{\Q}{\mathbb{Q}}

\newcommand{\bs}{\backslash}

\newcommand{\floor}[1]{\left \lfloor #1 \right \rfloor}

\newcommand\GG{{\mathcal G}}

\newcommand\LL{{\mathcal L}}

\newcommand\acl{\hbox{\rm acl}}

\newcommand\az{{\aleph_0}}

\newcommand\rn{{f_M(n)}} 
\newcommand\sub{{\varphi_M(n)}} 
 
\newcommand{\Wr}{\mathrm{\thinspace Wr \thinspace}}

\title[Monadic stability and growth rates]
 {Monadic stability and growth rates of $\omega$-categorical structures} 

\author{S. Braunfeld}

\classno{05E18, 03C15 (primary), 05A16 (secondary)}

\begin{document}
\maketitle

\begin{abstract}
For $M$ $\omega$-categorical and stable, we investigate the growth rate of $M$, i.e. the number of orbits of $Aut(M)$ on $n$-sets, or equivalently the number of $n$-substructures of $M$ after performing quantifier elimination. We show that monadic stability corresponds to a gap in the spectrum of growth rates, from slower than exponential to faster than exponential. This allows us to give a nearly complete description of the spectrum of slower than exponential growth rates (without the assumption of stability), confirming some longstanding conjectures of Cameron and Macpherson and proving the existence of gaps not previously recognized.
\end{abstract}


\section{Introduction}

 Generalizing the classic combinatorial problem of counting the orbits of a group acting on a finite set, Cameron began the study of counting the orbits of a group acting on a countably infinite set. In particular, the {\em growth rate} of $G$, i.e. the function $f_G(n)$ counting the number of orbits of $n$-sets, has received much attention under the  additional assumption it is always finite, which is equivalent to assuming the orbits are induced by $Aut(M)$ acting on some $\omega$-categorical $M$. The growth rate may also be viewed as counting the number of $n$-types with distinct entries, up to reordering the variables, or as counting the number of unlabeled isomorphism types of size $n$ in hereditary classes arising as the substructures of a homogeneous $\omega$-categorical relational structure.
 
 As is commonly observed in the growth rates of hereditary classes, there are gaps in the allowable asymptotic behavior of the function $\rn= f_{Aut(M)}(n)$. Our first result proves one such gap under the model-theoretic assumption that $M$ is stable, corresponding to whether or not $M$ is monadically stable. We say $f(n)$ is {\em slower than exponential} if $f(n) < c^n$ for every $c > 1$, and {\em faster than exponential} if $f(n)>c^n$ for every $c>1$ (here, and for the remainder of the paper, the bounds on growth rates only hold eventually).

\begin{theorem} \label{thm:1.1}
Suppose $M$ is $\omega$-categorical and stable. Then one of the following holds.
\begin{enumerate}
\item $M$ is monadically stable, and $\rn$ is slower than exponential.
\item $M$ is not monadically stable, and $\rn$ is faster than exponential. (More precisely, $f_M(n) > \floor{n/4}!$.)
\end{enumerate}
\end{theorem}

Model theoretic dividing lines have been shown to be significant for $\rn$ before; for example \cite{Mac3} studies the influence of the independence property. As these dividing lines were developed for counting problems related to infinite models, it is natural that they are relevant for counting finite models as well, particularly when phrased as a type-counting problem.

 Recent work of Simon \cite{Sim} reduces many questions about the behavior of $\rn$ to the stable case, and so we may confirm two longstanding conjectures of Macpherson and prove the existence of gaps not previously recognized.

\begin{theorem}[\cite{Mac1}*{Conjecture 3.2}]
Suppose $M$ is $\omega$-categorical and primitive, and $f_M(n)$ is not constant equal to 1. Then there is some polynomial $p(n)$ such that $f_M(n) > \frac{2^n}{p(n)}$.
\end{theorem}

\begin{theorem} \label{thm:intro phi growth}
Suppose $M$ is $\omega$-categorical and $\rn < \frac{\phi^n}{p(n)}$, for every polynomial $p(n)$, with $\phi$ the golden ratio. Then $\rn$ is slower than exponential and one of the following holds.
\begin{enumerate}
\item There are $c>0$, $k \in \N$ such that $\rn \sim cn^k$.
\item There are $c>0$, $k \in \N$ such that  $\rn =\exp\left((c+o(1))\left(n^{1-\frac1k}\right)\right)$
\item Let $\log^r(n)$ denote the $r$-fold iterated logarithm. There are $c>0$ and $k$, $r \in \N$ such that 
$ \rn = \exp\left((c+o(1))\left(\frac{n}{\left(\log^{r}(n)\right)^{1/k}}\right)\right)$
\end{enumerate}
Furthermore, in the first case all $k$ are achievable, in the second case all $k \geq 2$ are, and in the third case all $(r, k) \in (\N^+)^2$ are.
\end{theorem}

\begin{theorem}[\cite{Mac2}*{Conjecture 1.4}]
Suppose $M$ is $\omega$-categorical and $f_M(n)$ is not bounded above by a polynomial, but there is some $\epsilon > 0$ such that $\rn$ is bounded above by $e^{n^{1-\epsilon}}$. Then there is some $k \in \N$ such that, for any $\epsilon > 0$, 
 \[\exp\left(n^{(1-1/k) -\epsilon}\right) < \rn < \exp\left(n^{(1-1/k) +\epsilon}\right)\]
\end{theorem}

 We also slightly sharpen an earlier result of Macpherson to obtain a gap from polynomial growth to partition function growth (Corollary \ref{thm:poly part}). This gap corresponds to whether $M$ is cellular, which also corresponds to a gap for the analogous labeled enumeration question \cites{BB, LT}.

The proof of Theorem \ref{thm:1.1} rests on two model-theoretic results. That a monadically unstable structure has fast growth rate follows from a result of Baldwin and Shelah \cite{BS} that allows us to code bipartite graphs in a mild expansion of such a structure. That a monadically stable structure has slow growth rate follows from Lachlan's classification of $\omega$-categorical monadically stable structures as hereditarily cellular \cite{Lach}, which in turn depends on the results of \cite{BS}.
 
\subsection{Conventions and notation}
Unless otherwise stated, $M$ denotes a countable $\omega$-categorical structure in a countable relational language.

Unless otherwise stated, inequalities involving growth rates are to be understood as holding eventually.

Let $G$ be a permutation group acting on $X$, and $A \subset X$. Then $G_{(A)}$ denotes the pointwise stabilizer of $A$, while $G_{\set{A}}$ denotes the setwise stabilizer of $A$. Also, if $A$ is fixed setwise by $G$, then $G^A$ is the permutation group induced by $G$ on $A$.

$S_\infty$ denotes the automorphism group of a countable pure set.

\section{Monadic stability}

In this section, we introduce the results we need about monadic stability. First we give the Baldwin-Shelah characterization in terms of whether a theory admits coding. Then we give Lachlan's classification of $\omega$-categorical monadically stable structures.

\begin{definition}
A theory $T$ is {\em monadically stable} if every expansion of $T$ by unary predicates is stable. A structure $M$ (not necessarily $\omega$-categorical) is monadically stable if $Th(M)$ is.
\end{definition}

\begin{definition} \label{def:coding}
An $\LL$-theory $T$ (not necessarily $\omega$-categorical) {\em admits strongish coding} if there is some $\LL$-formula (with parameters) $\psi(a,b,c)$, some $M \models T$, and infinite disjoint $A,B,C \subset M$ such that $\psi$ defines the graph of a bijection from $A \times B$ to $C$, if we restrict to cases where $a \in A, b \in B$, and $c \in C$.
\end{definition}

The definition of {\em admits coding} from \cite{BS} allows $\psi$ to make use of predicates for the sets $A,B$, and $C$. There is also a definition of {\em admits strong coding}, which is too strong for our purposes.

\begin{remark} \label{rem:coding}
Suppose $T$ admits strongish coding, as witnessed by $M \models T$, $\psi$, and $A, B, C \subset M$. Expanding $M$ by the constants appearing in $\psi$, by unary predicates for $A,B,C$, and by another suitably chosen unary predicate $D \subset C$, we may define the edges of any bipartite graph on $A \times B$ by $E(a,b) \iff \exists (d \in D) \psi(a,b,d)$.
\end{remark}

\begin{theorem}[\cite{BS}*{Lemma 4.2.6}] \label{thm:ms coding}
Suppose $T$ is a countable theory that is stable but not monadically stable. Then $T$ admits strongish coding.
\end{theorem}

Although the conclusion of \cite{BS}*{Lemma 4.2.6} is only that $T$ admits coding, it is clear from the proof that $T$ admits strongish coding, as the formula used is an $\LL$-formula. Also, it is not immediate that the hypotheses of \cite{BS}*{Lemma 4.2.6} are equivalent to monadic instability, but this equivalence is discussed in the introduction to \cite{Lach}.

We now continue to Lachlan's classification of $\omega$-categorical monadically stable structures. There are two descriptions of such structures given in \cite{Lach}, and we give both, although we primarily use the latter. Our presentation and terminology occasionally differs slightly from \cite{Lach}. It may be helpful to consider Example \ref{ex:cell} when reading the following definition.

\begin{definition} \label{def:cell part}
A {\em cellular-like partition} of a structure $M$ is a triple $(K, E_0, E_1)$ satisfying the following.
\begin{enumerate}
\item $K$ is finite and fixed setwise by $Aut(M)$. (Note: If $M$ is $\omega$-categorical, then $K = \acl(\emptyset)$.)
\item $E_0, E_1$ are $Aut(M)$-invariant equivalence relations on $M \bs K$, with $E_1$ refining $E_0$.
\item There are finitely many $E_0$-classes.
\item Each $E_0$-class splits into infinitely many $E_1$-classes.
\item For each $E_0$-class $C$, and any $Y \subset C$ a union of $E_1$-classes, $Aut(M)_{(M\bs Y)}$ induces the full symmetric group on $Y/E_1$.
\end{enumerate}
\end{definition}

\begin{definition}
	Given an $\LL$-structure $M$ and $A \subset M$, the {\em structure induced on $A$ over its complement} is the structure with universe $A$ equipped with a relation for each $\LL$-definable relation using parameters from $M \bs A$.
	
Given a cellular-like partition of $M$, we define a {\em component} of $M$ to be the structure induced on an $E_1$-class of $M$ over its complement, or on $K$ over its complement.
\end{definition}

\begin{definition}
The class of {\em hereditarily cellular} structures is the least class containing all finite structures and that is closed under forming cellular-like partitions where each component is from the class.
\end{definition}

\begin{theorem}[\cite{Lach}] \label{thm:lach}
Suppose $M$ is hereditarily cellular.
\begin{enumerate}
\item The cellular-like partition witnessing that $M$ is hereditarily cellular is unique.
\item $M$ is $\omega$-categorical $\omega$-stable, with Morley rank equal to the depth defined in Definition \ref{def:depth}.
\end{enumerate}
\end{theorem}

\begin{definition} \label{def:depth}
Let $M$ be hereditarily cellular, as witnessed by the cellular-like partition $(K, E_0, E_1)$. We define $depth(M)$ as follows.
\begin{enumerate}
\item If $M$ is finite, $depth(M) = 0$.
\item Otherwise, $depth(M) = \max \set{depth(C)+1 | \text{$C$ is a component of $M$}}$.
\end{enumerate}

$M$ is {\em cellular} if it is hereditarily cellular of depth 1, i.e. every component is finite.
\end{definition}

\begin{example} \label{ex:cell}
Let $G_0$ be the graph consisting of a disjoint union of infinitely many edges. This is cellular with $K = \emptyset$, $E_0$ the one-class relation, and each edge an $E_1$-class. 

The graph $G_1$, obtained from $G_0$ by choosing a single point in each $E_1$-class and adding an edge between every pair of chosen points, is also cellular with the same cellular-like partition.

Let $G_2 = G_0 \sqcup G_1 \sqcup \set{x}$, with an edge added between $x$ and every point of $G_0$. This is cellular with $K = \set{x}$, $G_0$ and $G_1$ the $E_0$-classes, and $E_1$-classes as before.
\end{example}

\begin{example}
The simplest hereditarily cellular structure of depth 2 is an equivalence relation $E$ with infinitely many classes, each infinite. Here $K = \emptyset$, $E_0$ has one class, and $E_1=E$.
\end{example}

\begin{definition} \label{def:wreath}
Given an automorphism group $Aut(M)$ and $1 \leq k \leq \infty$, the {\em wreath product of $Aut(M)$ with $S_k$}, denoted $Aut(M) \Wr S_k$, is the automorphism group of the structure obtained by taking an equivalence relation with $k$ many classes and making each class a copy of $M$, with no additional structure between classes.

A more general definition, phrased solely in terms of groups, may be found in \cite{olig}*{\S 1.2}.
\end{definition}

\begin{example}
Let $G_0$ and $G_1$ be as in Example \ref{ex:cell}. Then $Aut(G_0) \cong \Z_2 \Wr S_\infty$. This is not the case for $Aut(G_1)$.
\end{example}

Lachlan defines the operations {\em loose union} and {\em $\omega$-stretch} on permutation groups. To a first approximation loose union should be considered direct product and $\omega$-stretch should be considered the wreath product with $S_\infty$. This is made more precise in Remark \ref{rem:loose fin} and Lemma \ref{lemma:stretch growth}. 

\begin{definition} \label{def:loose stretch}
For $k \in \N$ and each $i \in [k]$, let $G_i$ be a permutation group acting on $X_i$. Then $H$ acting on $Y=\sqcup_i X_i$ is a {\em loose union} of $\set{G_i}$ if it satisfies the following.
\begin{enumerate}
\item $\set{X_1, \dots, X_k}$ is $H$-invariant.
\item $G_i = H_{\set{X_i}}^{X_i}$
\item $\left[G_i : H_{(Y\bs X_i)}^{X_i}\right] < \infty$
\end{enumerate}

Let $G$ be a permutation group acting on $X$. Then $H$ acting on $Y$ is an {\em $\omega$-stretch} of $G$ if it satisfies the following.
\begin{enumerate}
\item There is an $H$-invariant equivalence relation $E$ on $Y$.
\item $X$ is an $E$-class.
\item $|Y/E| = \az$ and $H$ induces the full symmetric group on $Y/E$.
\item $G = H_{\set{X}}^X$
\item $\left[G: H_{(Y \bs X)}^X \right] < \infty$
\end{enumerate}

We let $\GG$ denote the set of all permutation groups obtained from finite permutation groups by a finite number of loose unions and $\omega$-stretches.
\end{definition}

\begin{lemma} \label{lemma:loose dir wr}
 For $k \in \N$ and each $i \in [k]$, let $G_i$ be a permutation group acting on $X_i$, and let $H$ be a loose union of $\set{G_i}$. Then (after reordering), there is some $j \leq k$ and $k_i \in \N$ for each $i \leq j$ such that $\Pi_{1\leq i \leq j} \left(H_{(Y\bs X_i)}^{X_i}\right)^{k_i} \leq H \leq \Pi_{1\leq i \leq j} G_i \Wr S_{k_i}$, where $\left(H_{(Y\bs X_i)}^{X_i}\right)^{k_i}$ is the $k_i$-fold direct product of $H_{(Y\bs X_i)}^{X_i}$ with itself.
\end{lemma}
\begin{proof}
We have that $H$ acts on $\set{X_1, \dots, X_k}$, so pick one representative from each orbit of this action, reorder so the representatives are $X_i$ for $i \leq j$, and let $k_i$ be the size of the orbit of $X_i$. Let $O_i$ be the union of all $X_\ell$ in the $H$-orbit of $X_i$, let $H_{\set{i}} = H_{\set{O_i}}^{O_i}$, and let $H_{(i)} = H_{(Y \bs O_i)}^{O_i}$. Then $H$ acts on $Y$ as some subgroup of $\Pi_{1 \leq i \leq j} H_{\set{i}}$ and some supergroup of $\Pi_{1\leq i \leq j} H_{(i)}$. Since $H_{\set{i}}$ and $H_{(i)}$ act on $O_i$ as subgroups of $G_i \Wr S_{k_i}$ and as supergroups of $\left(H_{(Y\bs X_i)}^{X_i}\right)^{k_i}$, we are finished.
\end{proof}

\begin{remark} \label{rem:loose fin}
In some sense, we would like to say that $H$ is a finite index subgroup of $\Pi_{1 \leq i \leq k}G_i$, but this is not quite true as $H$ might permute the $X_i$. But as $\Pi_{1\leq i \leq j} \left(H_{(Y\bs X_i)}^{X_i}\right)^{k_i}$ is a finite index subgroup of $\Pi_{1\leq i \leq j} G_i \Wr S_{k_i}$, we have that $H$ is as well.
\end{remark}

\begin{lemma}[\cite{Lach}*{Lemma 3.3}] \label{lemma:stretch growth}
Let $G \in \GG$ and $H$ be an $\omega$-stretch of $G$. Then, with notation as in Definition \ref{def:loose stretch}, there is $N \leq H$ of finite index such that $N$ acts as $H_{(Y \bs X)}^X \Wr S_\infty$.
\end{lemma}
While Lemma \ref{lemma:stretch growth} is not part of the statement of \cite{Lach}*{Lemma 3.3}, it is shown in the course of the proof of Part (i), Case 2. In fact, the proof shows that if we view $H$ as a cover of $S_\infty$, the cover splits.

\begin{theorem}[\cite{Lach}] \label{thm:G2}
 $M$ is $\omega$-categorical and monadically stable if and only if $M$ is hereditarily cellular if and only if $Aut(M) \in \GG$.
\end{theorem}

Furthermore, Lachlan also proves the following, which we will use to perform induction on hereditarily cellular structures.

\begin{lemma} [\cite{Lach}*{Theorem 3.4}] \label{lemma:ind}
Suppose $M$ is hereditarily cellular of depth $d>0$, with cellular-like partition $(K, E_0, E_1)$.
\begin{enumerate}
\item  If $K = \emptyset$ and $|M/E_0|=1$, then $Aut(M)$ is an $\omega$-stretch of the automorphism group of a hereditarily cellular structure of depth $d-1$.
\item Otherwise, let $X_1, \dots, X_k$ be an enumeration of $(\set{K} \cup \set{(M \bs K)/E_0}) \bs \set{\emptyset}$. Then $Aut(M)$ is a loose union of $Aut(M)_{\set{X_1}}^{X_1}, \dots, Aut(M)_{\set{X_k}}^{X_k}$.
 \end{enumerate}
\end{lemma}

\section{From slower than exponential to faster than exponential}

In this section, we prove Theorem \ref{thm:1.1}. It is slightly unexpected that $M$ being monadically stable should correspond to a dividing line in the behavior of $Age(M)$, i.e. the class of finite substructures of $M$. In particular, monadic stability seems like it might be too strong a condition, since unary predicates allow us to restrict to substructures, but can also do more. Some ideological explanation for this behavior is provided in \cite{BL}, where it is shown that if $M$ has quantifier elimination, then $Age(M)$ is stable if and only if $M$ is monadically stable. 

\begin{definition}
Let $G$ be a permutation group acting on a countable set $X$. Then $G$ acts on $n$-subsets of $X$ elementwise. The {\em growth rate of $G$} is the function $f_G(n)$, giving the number of orbits of $G$ on $n$-subsets.

In the case when $f_G(n)$ is always finite, after replacing $G$ with its closure with respect to the topology of pointwise convergence on $Sym(X)$, which doesn't affect its orbits on $n$-subsets, we may view $X$ as the universe of an $\omega$-categorical structure $M$, and $G$ as $Aut(M)$ with its natural action. In this case, the {\em growth rate of $M$}, denoted $\rn$, is just $f_{Aut(M)}(n)$.

If $M$ is any countable structure (not necessarily $\omega$-categorical), let $\sub$ (sometimes called the profile of $M$, or the unlabeled speed of $M$) count the number of substructures of $M$ of size $n$, up to isomorphism.
\end{definition}
\begin{remark}
If $M$ is $\omega$-categorical and has quantifier elimination, then $\rn = \sub$.

Also, $\rn$ and $\sub$ are non-decreasing (see \cite{olig}*{Chapter 3.1} and \cite{Pou}*{Theorem  4}).
\end{remark}

\begin{example} \label{ex:growth}
We provide some examples of growth rates. All $M$ will have quantifier elimination, so $\rn = \sub$.
\begin{enumerate}
\item $M = (\Q, \leq)$ and $\sub = 1$ for all $n$.
\item $M$ is an equivalence relation with infinitely many classes, each infinite. Then $\sub$ is the partition function, which is roughly $e^{n^{1/2}}$. 
\item $M = (X, E, \leq)$ where $E$ is an equivalence relation on $X$ with infinitely many classes, each of size 2, and $\leq$ is a dense linear order between $E$-classes (but not between points of $X$). Then $\sub$ is the $n^{th}$ Fibonacci number, and so roughly $\phi^n$.
\item $M$ is the Rado graph, and $\sub$ is roughly $2^{n^2/2}$.
\end{enumerate}
\end{example}

Theorem \ref{thm:1.1} follows quite rapidly from the results of the previous section. However, we must first establish bounds on how various operations affect growth rates, presumably none of which are original.

The following lemma is standard. For a proof, see \cite{Sim}*{Lemma 2.2}. 
\begin{lemma} \label{lemma:const}
Suppose $M^+$ is an expansion of $M$ obtained by naming finitely many constants. Then there is a polynomial $p(n)$ such that $\rn > \frac{f_{M^+}(n)}{p(n)}$.
\end{lemma}

\begin{lemma} \label{lemma:unary growth}
Suppose $M^+$ is an expansion of $M$ (not necessarily $\omega$-categorical) by $k$ unary predicates. Then, for every $n$, $\varphi_M(n) \geq \frac{\varphi_{M^+}(n)}{(2^{k})^n}$.
\end{lemma}
\begin{proof}
Given a substructure $X \subset M$ of size $n$, every point in the expansion is put into one of the $2^k$ subsets of unary predicates. Thus there are at most $(2^k)^n$ possible expansions of $X$.
\end{proof}

\begin{lemma} \label{lemma:fin index}
Let $G$ be a permutation group acting on $X$, and let $H \leq G$ be a subgroup of index $k$. Then, for every $n$, $f_G(n) \leq f_H(n) \leq k \cdot f_G(n)$.
\end{lemma}
\begin{proof}
The lower bound is immediate since $H \leq G$. For the upper bound, fix $n$. It suffices to show each $G$-orbit divides into at most $k$ $H$-orbits, so we work one orbit at a time. Let $O $ be an orbit of the action of $G$ on $n$-subsets of $X$. Let $g_1, \dots, g_k$ be representatives for the right cosets of $H$. Pick $Y \in O$, and let $Y_i = g_iY$. For any $Z \in O$, there is some $g \in G$ such that $Z = gY$, and so there is some $h \in H$ and some $i \leq k$ such that $Z = (hg_i)Y = hY_i$. Thus $O$ divides into at most $k$ $H$-orbits.
\end{proof}

\begin{lemma} \label{lemma:dir prod}
Let $G$ act on $X$ and $H$ act on $Y$, both countably infinite sets. Letting the direct product $G \times H$ act naturally on $X \sqcup Y$, then for every $n$, $f_{G \times H}(n) \leq (n+1) f_G(n)f_H(n)$. 

If $|Y|=k$ is finite, then for every $n$, $f_{G \times H}(n) \leq (k+1) f_G(n) \max_{\ell}(f_H(\ell))$.
\end{lemma}
\begin{proof}
Choose $n_1, n_2$ so that $n = n_1+n_2$. Then the number of orbits of $G \times H$ on $n$-sets with $n_1$ points in $X$ and $n_2$ points in $Y$ is $f_G(n_1)f_H(n_2)$. As there are $n+1$ choices for $n_1, n_2$, and $f_G(n)$  and $f_H(n)$ are non-decreasing, we are finished.

The argument is similar for $Y$ finite, except we no longer have that $f_H(n)$ is non-decreasing.
\end{proof}

\begin{lemma} \label{lemma:loose growth}
Let $H$ be a loose union of $G_1, \dots, G_k$, acting on $X_i$ for $i \in [k]$, respectively. Suppose $X_i$ is finite for $i > j$. Then there is some $c \in \R$ such that $f_H(n) \leq (cn^k) \Pi_{i \in [j]} f_{G_i}(n)$.
\end{lemma}
\begin{proof}
By definition, $H$ is a supergroup of $\Pi_{i \in [k]} H_{(Y\bs X_i)}^{X_i}$. Recall $H_{(Y\bs X_i)}^{X_i}$ must be a finite index subgroup of $G_i$. The result then follows from Lemmas \ref{lemma:dir prod} and \ref{lemma:fin index}. 
\end{proof}

\begin{lemma}[\cite{olig}*{3.8}] \label{lemma:wr prod}
Let $G$ be a permutation group such that $f_G(n)$ is slower than exponential. Then $f_{G \Wr S_\infty}(n)$ is also slower than exponential.
\end{lemma}

\begin{lemma} \label{lemma:ms}
Suppose $T$ (not necessarily $\omega$-categorical) has quantifier elimination and admits strongish coding. Then for any $M \models T$, $\sub$ grows faster than exponential. (More precisely, $\varphi_M(n) > \floor{n/4}!$.)
\end{lemma}
\begin{proof}
Let $N\models T$ and $\psi$ be a quantifier-free formula such that $N$ and $\psi$ witness that $T$ admits strongish coding. Since $\varphi_N(n) = \sub$ for any $M \models T$, it suffices to show the conclusion for $N$. Let $N^+$ be the expansion of $N$ by finitely many constants and by four unary predicates $A,B,C,D$ as in Remark \ref{rem:coding}, such that $\psi$ defines the random bipartite graph on $A \times B$. Without loss of generality, we may assume $A \cup B \cup C$ is coinfinite. Given a bipartite graph $G$ with distinguished parts, with $n$ edges, and no isolated vertices (and thus at most $n$ vertices in each part), we produce a substructure of size $3n$ by choosing $A' \subset A, B' \subset B$ such that $\psi$ defines $G$ on $A' \times B'$, as well as the elements of $D$ needed to witness the edges; if $G$ has fewer than $n$ vertices in either part, we also choose elements outside of $A \cup B \cup C$ to bring the total to $3n$. For non-isomorphic bipartite graphs, the structures thus produced will be non-isomorphic.


By \cite{CPS}*{Proposition 7.1}, the number of bipartite graphs with distinguished parts, $n$ edges, and no isolated vertices (there called $F_{0101}$) is bounded below by $\left(\frac{n}{(\log(n))^{2+\epsilon}}\right)^n$ for every $\epsilon > 0$, which is thus a lower bound for $\varphi_{N^+}(3n)$. By Lemmas \ref{lemma:const} and \ref{lemma:unary growth}, we get the same bound for $\varphi_N(3n)$ by absorbing lower-order terms into the $\epsilon$. By Stirling's approximation, we get the desired bound for $\varphi_N(n)$.
\end{proof}

\begin{theorem} \label{thm:main}
Suppose $M$ is stable. Then one of the following holds.
\begin{enumerate}
\item $M$ is monadically stable, and $\rn$ is slower than exponential.
\item $M$ is not monadically stable, and $\rn$ is faster than exponential. (More precisely, $f_M(n) > \floor{n/4}!$.)
\end{enumerate}
\end{theorem}
\begin{proof}
(i) By Theorem \ref{thm:G2}, it suffices to show that groups of slower than exponential growth rate are closed under loose union and $\omega$-stretch.

For loose union, we are finished by Lemma \ref{lemma:loose growth}.

Now suppose $f_G(n)$ is slower than exponential, and $H$ is an $\omega$-stretch of $G$. By Lemma \ref{lemma:stretch growth}, there is $G_0 \leq G$ of finite index such that $H$ is a supergroup of $G_0 \Wr S_\infty$. We are then finished by Lemmas \ref{lemma:fin index} and \ref{lemma:wr prod}.

(ii) By Theorem \ref{thm:ms coding}, $Th(M)$ admits strongish coding. As Morleyizing (i.e. adding a relation symbol encoding every formula) does not change $\rn$, we may assume $Th(M)$ has quantifier elimination, so $\rn = \sub$. We are then finished by Lemma \ref{lemma:ms}.
\end{proof}

\begin{remark}
The $\omega$-categorical structure $M$ whose automorphism group is $S_\infty \Wr S_2$ with its product action (different from Definition \ref{def:wreath}; see \cite{olig}*{\S 3.3, Example 4}) is stable and has growth rate equal to the number of bipartite graphs with $n$ edges and no isolated vertices. Thus the lower bound in Theorem \ref{thm:main}(ii) cannot be improved beyond the number of such graphs. Precise asymptotics for this problem are unknown, although the number is below $n!$.
\end{remark}

\section{Growth rates of hereditarily cellular structures}
In this section, we provide bounds on the growth rates of hereditarily cellular structures, depending on their depth. In particular, we see that the depth of a hereditarily cellular structure is recoverable from its growth rate.

We first look at cellular structures. If $G$ is the automorphism group of a finite structure of size $k$, then $f_{G \Wr S_\infty}(n)$ is bounded above by the product of the number of ways to partition $n$ points into parts of size $k$ and into at most $k$ parts, which are both polynomial. Then by Lemma \ref{lemma:stretch growth},  any $\omega$-stretch of $G$ has growth rate bounded above by a polynomial. By Lemma \ref{lemma:loose growth}, loose unions preserve this property, so it follows for cellular structures.  

\begin{theorem} [\cite{FT}] \label{lemma:cell growth} 
Suppose $\rn$ is bounded above by a polynomial. Then there are $c>0, k \in \N$ such that $\rn \sim cn^k$. In particular, the conclusion holds for cellular structures.
\end{theorem}

If $M$ is a hereditarily cellular structure of depth $d$, we will see that the growth rate of $Aut(M)$ will essentially be that of $Aut(N) \Wr S_\infty$, for $N$ some hereditarily cellular structure of depth $d-1$. By \cite{olig}*{3.7}, for any $G$, the generating function for $f_{G \Wr S_\infty}(t)$ has the product form $\Pi_{n \geq 1}(1-t^n)^{-f_G(n)}$. Applying another result on asymptotics for such products yields the following.

\begin{theorem}[\cite{gpart}*{Theorem 2.1}, \cite{olig}*{3.7}] \label{thm:gpart}
Let $G = H \Wr S_\infty$, and let $\log^r(n)$ denote the $r$-fold iterated logarithm. 
\begin{enumerate}
\item If $f_H(n) = (a + o(1))n^k$, then there is a constant $b=b(a,k)$ such that $f_G(n) = \exp\left((b+o(1))n^{1-\frac1{k+2}}\right)$.
\item If $f_H(n) = \exp\left((b+o(1))n^{1-\frac1{k}}\right)$, then there is a constant $c = c(b,k)$ such that $f_G(n) = \exp\left((c+o(1))\frac{n}{(\log(n))^{1/(k-1)}} \right)$.
\item If $f_H(n) = \exp\left((c+o(1))\frac{n}{\left(\log^r(n)\right)^{1/k}} \right)$, then $f_G(n) = \exp\left((c+o(1))\frac{n}{\left(\log^{(r+1)}(n)\right)^{1/k}} \right)$.
\end{enumerate}
\end{theorem}

\begin{lemma} \label{lemma:d2 growth}
Suppose $M$ is hereditarily cellular of depth 2. Then there are $c>0, k \in \N$ such that
 \[\rn = \exp\left((c+o(1))\left(n^{1-\frac1k}\right)\right)\]
\end{lemma}
\begin{proof}
We proceed by induction on the construction of $Aut(M)$ from finite groups, using Lemma \ref{lemma:ind}.

For the base case, $Aut(M)$ is the $\omega$-stretch of $Aut(M')$, for some cellular $M'$. By Lemma \ref{lemma:stretch growth}, thus there is some $N \leq Aut(M)$ of finite index and $H \leq Aut(M')$ of finite index such that $N$ acts as $H \Wr S_\infty$. By Lemma \ref{lemma:fin index}, it suffices to prove the result for $N$. 
As $M'$ is cellular, its growth rate is bounded above by a polynomial, and thus so is the growth rate of $H$ by Lemma \ref{lemma:fin index}. By Theorem \ref{lemma:cell growth}, there exist $a >0, d \in \N$ such that $f_H(n) = (a+o(1))n^d$. Then by Theorem \ref{thm:gpart}(i), $f_N(n)$ is as desired.

For the inductive step, there is some $n \in \N$ such that $Aut(M)$ is the loose union of $\set{Aut(N_i) | i \in [n]}$, where each $N_i$ is hereditarily cellular of depth at most 2. Furthermore, for any $N_i$ of depth $2$, let $(K_i, E_{i,0}, E_{i,1})$ be its cellular-like decomposition; then $K_i = \emptyset$ and there is only one $E_{i,0}$-class. Then by the base case, the result holds for any $N_i$ of depth 2. Otherwise $N_i$ is cellular, and so has polynomial growth by Lemma \ref{lemma:cell growth}, or is finite. Regardless, the result follows by applying Lemma \ref{lemma:loose growth}.
\end{proof}

\begin{lemma} \label{lemma:d3 growth}
Suppose $M$ is hereditarily cellular of depth $d \geq 3$. Then there are $c>0, k \in \N$ such that, letting $\log^r(n)$ denote the $r$-fold iterated logarithm,
\[\rn = \exp \left((c+o(1))\left(\frac{n}{(\log^{d-2}(n))^{1/k}}\right)\right)\]
 
\end{lemma}
\begin{proof}
We proceed by induction on $d$. For both the base case and inductive step, we perform another induction on the construction of $Aut(M)$ from finite groups. The proof is as in Lemma \ref{lemma:d2 growth}, except we use the relevant part of Theorem \ref{thm:gpart}.
\end{proof}

\section{The spectrum of slower than exponential growth rates}

The following two theorems of Simon reduce many questions about growth rates of $\omega$-categorical $M$ to the case where $M$ is stable.

\begin{notation}
We let $\phi$ denote the golden ratio, so $\phi \approx 1.618$.
\end{notation}

\begin{theorem}[\cite{Sim}*{Theorem 1.6}] \label{thm:S 1.6}
Suppose $M$ is such that there is no polynomial $p(n)$ such that $\rn \geq \frac{\phi^n}{p(n)}$. Then there is a stable reduct $M_*$ of $M$ such that $f_{M_*}(n) = \rn$ for all $n$.
\end{theorem}

The point of the hypothesis of Theorem \ref{thm:S 1.6} is that $M$ should not encode Example \ref{ex:growth}(3), even over finitely many parameters.

\begin{definition}
$M$ is {\em primitive} if it has no $\emptyset$-definable equivalence relation, other than equality and the one-class relation (or equivalently, if $Aut(M)$ is primitive as a permutation group).
\end{definition}

\begin{theorem}[\cite{Sim}*{Theorem 1.3}] \label{thm:S 1.3}
Suppose $M$ is primitive, $\rn$ is not constant equal to 1, and there is no polynomial $p(n)$ such that $\rn \geq \frac{2^n}{p(n)}$. Then $M$ must be stable, but not $\omega$-stable.
\end{theorem}

The following theorem was initially proven by Macpherson \cite{Mac1} with the constant $c = 2^{1/5}$ in place of $2$. This was then improved to $c \approx 1.324$ by Merola \cite{Mer}, and then to $c \approx 1.576$ by Simon \cite{Sim}. The bound $c=2$ is optimal, as shown by the ``local order'' $S(2)$ (see \cite{olig}*{\S 3.3}).

\begin{theorem}[\cite{Mac1}*{Conjecture 3.2}] \label{thm:prim}
Suppose $M$ is primitive and $f_M(n)$ is not constant equal to 1. Then there is some polynomial $p(n)$ such that $f_M(n) > \frac{2^n}{p(n)}$.
\end{theorem}
\begin{proof}
Suppose $M$ is a counterexample. By Theorem \ref{thm:S 1.3}, $M$ must be stable. By Theorem \ref{thm:main}, $M$ must be monadically stable, and so hereditarily cellular. But the only infinite primitive hereditarily cellular structure is a pure set.
\end{proof}
\begin{remark}
We could also use the fact that Theorem \ref{thm:S 1.3} shows $M$ cannot be $\omega$-stable, and then apply Theorem \ref{thm:lach}.
\end{remark}

We now show there is a gap from slower than exponential growth rates to roughly $\phi^n$, and almost completely characterize the spectrum of slower than exponential growth rates.

\begin{theorem} \label{thm:phi growth}
Suppose $M$ is $\omega$-categorical and $\rn < \frac{\phi^n}{p(n)}$, for every polynomial $p(n)$. Then $\rn$ is slower than exponential and one of the following holds.
\begin{enumerate}
\item There are $c>0$, $k \in \N$ such that $\rn \sim cn^k$.
\item There are $c>0$, $k \in \N$ such that  $\rn =\exp\left((c+o(1))\left(n^{1-\frac1k}\right)\right)$
\item Let $\log^r(n)$ denote the $r$-fold iterated logarithm. There are $c>0$ and $k$, $r \in \N$ such that 
$ \rn = \exp\left((c+o(1))\left(\frac{n}{\left(\log^{r}(n)\right)^{1/k}}\right)\right)$
\end{enumerate}
Furthermore, in the first case all $k$ are achievable, in the second case all $k \geq 2$ are, and in the third case all $(r, k) \in (\N^+)^2$ are.
\end{theorem}
\begin{proof}
By Theorem \ref{thm:S 1.6}, we may assume $M$ is stable, and so $M$ must be hereditarily cellular. The result then follows from Theorem \ref{lemma:cell growth} and Lemmas \ref{lemma:d2 growth} and \ref{lemma:d3 growth}.

In the first case we may achieve any $k$ by taking the structure with $k+1$ unary relations, each infinite, that partition the domain, whose growth rate is $\binom{n+k}{k}$. The remaining cases then follow from Theorem \ref{thm:gpart} by taking iterated wreath products with $S_\infty$.
\end{proof}

In particular, we confirm the following conjecture of Macpherson.

\begin{corollary}[\cite{Mac2}*{Conjecture 1.4}] \label{thm:pfe}
Suppose $M$ is such that $f_M(n)$ is not bounded above by a polynomial, but there is some $\epsilon > 0$ such that $\rn$ is bounded above by $e^{n^{1-\epsilon}}$. Then there is some $k \in \N$ such that, for any $\epsilon > 0$
 \[\exp\left(n^{(1-1/k) -\epsilon}\right) < \rn < \exp\left(n^{(1-1/k) +\epsilon}\right)\]
\end{corollary}

We also slightly sharpen \cite{Mac2}*{Theorem 1.2}.

\begin{corollary} \label{thm:poly part}
Either $f_M(n)$ is bounded above by a polynomial or bounded below by the partition function.
\end{corollary}
\begin{proof}
	Again, it suffices to consider the hereditarily cellular structures. If $M$ is cellular then $f_M(n)$ is bounded above by a polynomial. Otherwise $M$ has a $\emptyset$-definable equivalence relation with infinitely many infinite classes, so its growth rate is at least the partition function.
\end{proof}

\section{Questions}

We have noted the growth rate of an $\omega$-categorical structure is the same as the function $\sub$ counting the $n$-substructures of a homogeneous $\omega$-categorical structure. However, we may consider $\sub$ for an arbitrary countable structure (e.g., see \cite{Pou}), or even more generally count the isomorphism types of size $n$ in an arbitrary hereditary class (e.g., see \cite{Kla}). For slow growth rates, the spectrum of growth rates seems to be approximately the same for all three questions. For example, the analogue of Corollary \ref{thm:pfe} holds for hereditary classes of graphs \cite{BBSS}.

\begin{question}
Do corollaries \ref{thm:pfe} and \ref{thm:poly part} hold for the growth rate of a hereditary class of structures in a finite relational language? If not, do they hold for $\sub$ for arbitrary countable $M$?
\end{question}

Even in the $\omega$-categorical case, some questions remain about structures with slower than exponential growth rate. First, although classifying the stable structures with slower than exponential growth suffices to understand the spectrum of growth rates, we may still try to classify such structures without assuming stability. For one example, consider an equivalence relation $E$ with infinitely many classes, each refined into infinitely many infinite $E'$-classes, and then expand by a relation equipping each $E'$-class with a dense linear order. For another example, consider an equivalence relation with two infinite classes, each equipped with a betweenness relation, such that automorphisms must preserve or reverse the order on both classes simultaneously. The proofs of \cite{Sim}*{Theorem 5.1, Corollary 5.5} suggest the following conjecture, although the structure $M_*$ there might be a reduct of the $M_0$ below.

\begin{conjecture} \label{conj:class}
The growth rate of $M$ is slower than exponential if and only if, up to interdefinability, $M$ is a reduct of an expansion of a hereditarily cellular structure $M_0$ by a binary relation $<$ satisfying the following conditions.
\begin{enumerate}
\item There is a $\emptyset$-definable equivalence relation $E$ on the reduct $M_0$ such that every $E$-class $C$ is either infinite with $Aut(M_0)_{(M_0 \bs C)}^C = S_\infty$ or a singleton.
\item If $M \models x < y$ then $x, y$ are contained in a single infinite $E$-class.
\item The restriction of $<$ to any infinite $E$-class is a dense linear order.
\end{enumerate}
\end{conjecture}

 There are finer gaps in the spectrum of growth rates than those described by Theorem \ref{thm:phi growth}. For example, in \cite{olig}*{\S 3.6} it is mentioned that the growth rate is either eventually constant or at least linear with slope $1/2$, and the following question is posed in the polynomial case.

\begin{problem} \label{prob:spec}
Given $k$ and $r$, which values of $c$ are achievable in each case of Theorem \ref{thm:phi growth}?
\end{problem}

In fact, the polynomial case of Problem \ref{prob:spec} suffices, since the constants $b(a,k)$ and $c(b,k)$ in Theorem \ref{thm:gpart} are specified in \cite{gpart}*{Theorem 2.1}.

The orbit algebra of \cite{alg} encodes much information about the growth rate, and is a significant tool for the analysis of the polynomial case in \cite{FT}, where the algebra is shown to be Cohen-Macaulay for cellular structures.

\begin{problem} \label{prob:alg}
Investigate the orbit algebras of hereditarily cellular structures.
\end{problem}

The next natural range of growth rates to consider are the exponential growth rates, i.e. those bounded above by $c^n$ for some $c \in \R$.

\begin{problem} \label{prob:exp}
Classify the $\omega$-categorical $M$ with $\rn < c^n$ for some $c \in \R$. Describe the corresponding spectrum of growth rates.
\end{problem}

We repeat two questions of Macpherson concerning the structures from Problem \ref{prob:exp}. The first asks whether, for homogeneous structures, such $M$ are exactly those such that $Age(M)$ contains no infinite antichains under embeddability \cite{MacHom}*{Question 2.2.7}. The second asks, given such an $M$ with $Aut(M)$ primitive and not highly transitive, whether $Aut(M)$ is contained in a Jordan group that is not highly transitive \cite{MacJord}*{Problem 8.1}.

We also conjecture the following, and note a tree decomposition for monadically NIP structures is conjectured in \cite{BS} similar to the one provided there for the monadically stable case.

\begin{conjecture}
If $M$ is $\omega$-categorical, then $\rn < c^n$ for some $c \in \R$ if and only if $M$ is monadically NIP.
\end{conjecture}

In \cite{BL}, it is shown that if $M$ is not monadically NIP then its growth rate is faster than exponential, and partial progress is made towards the question of Macpherson about infinite antichains.

The subcase of Problem \ref{prob:exp} for $M$ with $\rn < \frac{2^n}{p(n)}$ for every polynomial $p(n)$ is within more immediate reach. By \cite{Sim}*{Theorem 5.1}, these structures are obtained from the structures with slower than exponential growth by taking finite covers of the sets equipped with unstable reducts of dense linear orders. We define the generalized Fibonacci numbers of order $k$ by the recurrence $F_{n,k} = \sum_{i=1}^{i=k} F_{n-i,k}$, with $F_{n,k} = 0$ for $n < 0$ and $F_{0,k}=1$, so $\lim_{n \to \infty} F_{n,k}^{1/n}$ is the reciprocal of the smallest real root of $1-x-\dots-x^k$. Generalizing Example \ref{ex:growth}(3), $F_{n,k}$ is the growth rate of the analogous structure but with classes of size $k$.

\begin{conjecture}
If $M$ is $\omega$-categorical  and $\rn < \frac{2^n}{p(n)}$ for every polynomial $p(n)$, then there is some $k \in \N$ such that $\lim_{n \to \infty} F_{n,k}^{1/n} = \lim_{n \to \infty} \rn^{1/n}$.
\end{conjecture}

As mentioned before Theorem \ref{thm:prim}, the local order $S(2)$ shows the constant $c=2$ there is optimal. The following conjecture is mentioned in \cite{Sim}, and is another step toward Problem \ref{prob:exp}.

\begin{conjecture} \label{conj:mac}
Suppose $G$ is primitive and there are polynomials $p(n), q(n)$ such that $\frac{2^n}{p(n)} < f_G(n) < \frac{2^n}{q(n)}$. Then $G$ is either $Aut(S(2))$, or the group of automorphisms and anti-automorphisms of $S(2)$.
\end{conjecture}

\begin{acknowledgements}\label{ackref}
I thank Bertalan Bodor for pointing out errors in earlier versions and helping to formulate Conjecture \ref{conj:class}, Gregory Cherlin, Justine Falque, and Chris Laskowski for helpful discussions and feedback, Dugald Macpherson for suggesting Problem \ref{prob:alg},  Benyamin Riahi for helping the exposition, Pierre Simon for comments on questions posed in an earlier version, and the referee for helping the exposition.
\end{acknowledgements}

\affiliationone{
   S. Braunfeld\\
   William E. Kirwan Hall\\
      College Park, MD 20742\\
   USA
   \email{sbraunf@umd.edu
   }}

\end{document}